\documentclass{amsart}
\usepackage{graphicx,amssymb}
\newtheorem{thm}{Theorem}[section]
\newtheorem{cor}[thm]{Corollary}
\newtheorem{lem}[thm]{Lemma}
\newtheorem{prop}[thm]{Proposition}
\theoremstyle{definition}

\theoremstyle{remark}

\theoremstyle{question}
\newtheorem{qu}[thm]{Question}
\numberwithin{equation}{section}

\begin{document}

\title[Determinants of adjacency matrices of graphs]{Determinants of adjacency matrices of graphs}%
\author{Alireza Abdollahi}%
\address{Department of Mathematics, University of Isfahan, Isfahan 81746-73441, Iran}%
\email{a.abdollahi@math.ui.ac.ir}%
\thanks{}

\begin{abstract}
We study the set of  all determinants of adjacency matrices of graphs with a given number of vertices.
\end{abstract}
\maketitle
\section{\bf Introduction}

Let $G$ be a simple graph with finite number of vertices. We denote by $\det(G)$ the determinant of an adjacency matrix of $G$. This number $\det(G)$ is an integer and is an invariant of $G$ so that its value is independent of the choice of vertices in an adjacency matrix. 
In this paper, we study the distributions of $\det(G)$ whenever $G$ runs over graphs with finite $n$ vertices for a given integer $n\geq 1$. We denote by $\mathcal{G}_n$ the set of all non-isomorphic graphs with $n$ vertices and 
$$\mathcal{DG}_n=\big\{\det(G) \;|\; G\in \mathcal{G}_n\big\}, \;\; \alpha_n=\max \mathcal{DG}_n \;\;\text{and}\;\; \beta_n=\min \mathcal{DG}_n.$$

In \cite{FvdD}, Fallat  and van den Driessche studied, among others things, the maximum $W(n,k)$ and  minimum $w(n,k)$ of non-zero absolute values of   determinants of $k$-regular graphs with $n$ vertices and they determined $W(n,2)$, $W(n,n-3)$ and $w(n,2)$,  $w(n,n-3)$.\\
 
 It is a famous  result due to Hadamard \cite{Ha} that
if $A=[a_{ij}]$ is an $n\times n$ complex matrix such that  $||a_{ij}||\leq \mu$ for all $i,j$  then
$||\det(A)||\leq \mu^n n^{\frac{n}{2}}$.
In \cite{R},  Ryser found  an upper bound for the absolute value of a $(0,1)$ square matrix of size $n$ with $t$ non-zero entries.
 One may use Ryser's result to give an upper bound for the absolute values of the determinants of graphs with $n$ vertices and $m$ edges (see Theorem \ref{Ry}, below).\\
 
Using Brendan McKay's nauty, we have  computed $\mathcal{DG}_n$ for all $n\leq 9$ (see Proposition \ref{t}, below).\\
 
 Newman \cite[Theorem 2.2]{Ne} proved that $k\cdot\gcd(k,n)$ divides the determinants of $n\times n$ $(0,1)$ matrices whose all row and column sums equal to $k$. A similar result for the determinants of graphs is proved in Theorem \ref{Nee}.\\
 
 Hu \cite{H} has determined the determinant of graphs with exactly one cycle.
     Here we obtain the possible determinants of graphs with exactly  two cycles (see Proposition \ref{CC2}, below). 
\section{\bf Preliminaries}
\begin{prop}\label{KPC}
Let $K_n$ and $P_n$ be the complete graph and the  path with $n$ vertices, respectively, and $C_m$ be the cycle with $m\geq 3$ vertices.
\begin{enumerate}
\item $\det(K_n)=(-1)^{n-1}(n-1)$.
\item $\det(P_n)=\begin{cases} 0 & \mathrm{if} \; n \; \mathrm{is \; odd}\\
1 &\mathrm{if} \; n=2k, \; k \;\; \mathrm{is \; even}\\  -1 & \mathrm{if} \; n=2k, \; k \;\; \mathrm{is\; odd}  \end{cases}.$
\item $\det(C_m)=\begin{cases} 2 & \mathrm{if} \; m \; \mathrm{is \; odd}\\
0 &\mathrm{if} \; m=2k, \; k \;\; \mathrm{is \; even}\\  -4 & \mathrm{if} \; m=2k, \; k \;\; \mathrm{is\; odd}  \end{cases}.$
\end{enumerate}
\end{prop}
\begin{proof}
(1) \; It is well-known.\\
(2) \; It follows from the fact that $\det(P_n)=-\det(P_{n-2})$.\\
(3) \;  It follows from \cite[Theorem 1.3, p. 32]{CDS}.
\end{proof}
\begin{prop}\label{t} The following table gives the set of all determinants of graphs with at most $9$  vertices. By $i^j$ in  the $n$-th row of the
table, we mean that there are exactly $j$ non-isomorphic graphs with $n$ vertices whose determinants are all equal to $i$, and a single number $i$
 (with no exponent) shows that there is a unique graph with $n$ vertices whose determinant is equal to $i$.
\begin{center}
\begin{tabular}{|c|c|}
  \hline
  Number of Vertices & Determinants of graphs with multiplicities \\
  \hline
    $1$ & $0$ \\
  \hline
  $2$ & $0$,$-1$ \\
  \hline
  $3$ & $0^3$,$2$ \\
  \hline
  $4$ & $-3$,$0^7$,$1^3$ \\
  \hline
  $5$ & $-4$,$0^{25}$,$2^6$,$4$ \\
  \hline
  $6$ & $-5^3$,$-4^5$,$-1^{32}$,$0^{99}$,$3^{10}$,$4^2$,$7^2$ \\
  \hline
  $7$ & $-12^2$,$-10^2$,$-6^{13}$,$-4^{21}$,$-2^{20}$,$0^{690}$,$2^{204}$,$4^{40}$,$6^{17}$,$8^{25}$,$10^5$,$12^5$ \\
  \hline
    $8$ & $-28^2$,$-27^2$,$-24^5$,$-23^5$,$-20^7$,$-19^{21}$,$-16^{51}$,$-15^{43}$,$-12^{90}$,$-11^{79}$,$-8^{128}$, \\
    & $-7^{251}$,
      $-4^{581}$,$-3^{813}$,$0^{6551}$,$1^{2416}$,$4^{758}$,$5^{240}$,$8^{73}$,$9^{139}$,$12^{24}$,$13^{23}$,$16^{32}$,$17^8$,$20^1$,$21^3$\\
  \hline
  $9$ &   $-128^2$, $-96^3$,$-72^{12}$,$-64^7$,$-60^5$,$-56^{17}$,$-54^{12}$,$-50^{27}$,$-48^{13}$, $-46^{20}$,$-44^{39}$,$-42^{47}$,\\ & $-40^{103}$,$-38^{52}$,$-36^{110}$,$-34^{128}$,${-32}^{593}$,$-30^{199}$,$-28^{295}$,$-26^{392}$,$-24^{765}$,$-22^{579}$,\\ &
  $-20^{869}$,$-18^{2747}$,$-16^{2247}$,$-14^{1805}$,$-12^{3062}$,$-10^{4290}$,$-8^{17582}$,$-6^{8531}$,$-4^{14901}$, \\ &
  $-2^{57065}$,$0^{133174}$,$2^{6767}$,$4^{6950}$,$6^{4669}$,$8^{1566}$,$10^{1349}$,$12^{1156}$,$14^{695}$,$16^{606}$, \\ &
  $18^{106}$,$20^{297}$, $22^{173}$,$24^{240}$,$26^{95}$,$28^{91}$,$30^{61}$,$32^{46}$,$34^5$,$36^{32}$,$38^{28}$,$40^{3}$,
  $42^{17}$,$44^{16}$,$54^3$,$60^3$,$64$\\
    \hline
\end{tabular}
\end{center}
\end{prop}
\begin{proof}
Using {\sf nauty} of Brendan McKay, one can generate all graphs with at most $9$ vertices and then by  {\sf GAP} for example, it is easy to find the determinants.
\end{proof}
 Here are some questions which are motivated by the data given in Proposition \ref{t}.
\begin{prop}
\begin{enumerate}
\item   $\beta_n\not=0$ if and only if $\beta_n<0$ if and only if  $n\not\in\{1,3\}$.
\item $\alpha_n\not=0$ if and only if $\alpha_n>0$ if and only if $n\not\in\{1,2\}$.
\item If  $G$ is a graph with odd number of vertices, then $\det(G)$ is  even.
\item The determinant of every bipartite graph with odd vertices is zero.
\end{enumerate}
\end{prop}
\begin{proof}
(1) \; It is clear that if $n=1$ or $3$, then $\beta_n=0$ and if $n=2$, then $\beta_n=-1$. Thus assume that $n\geq 4$. It is enough to show that $\beta_n<0$. If $n$ is odd, we have $\beta_n\leq \det(C_{n-2}\dotplus P_2)=-2$, by Lemma \ref{KPC}. If $n$ is even, we have $\beta_n\leq \det(K_n)=-(n-1)$, by Lemma \ref{KPC}. This proves (2).\\
(2) \; If $n=1$ or $2$, then clearly $\alpha_n=0$. Thus it is enough to prove that $\alpha_n>0$ whenever $n\geq 3$. If $n$ is odd, then $\alpha_n\geq \det(K_n)=n-1$, by Lemma \ref{KPC}. If $n$ is even, then $\alpha_n\geq \det(K_{n-2}\dotplus P_2)=n-3$. This completes the proof of (3).\\
(3) \; Let $n=|V(G)|$ be odd and let $\mathcal{F}$ be an elementary figure  (if exists) of $G$ with $n$ vertices. Then, by definition, $\mathcal{F}$ is a disjoint union of $s$ number of edges and $t$ number of cycles. If $t=0$, then  $n=2s$ which is not possible.  It follows that any elementary figure of $G$ with $n$ vertices have at least one cycle. Now it follows from \cite[Theorem 1.3, p. 32]{CDS} that $\det(G)$ is an even number.\\
(4) \; Let $G$ be a bipartite graph with odd $n$ vertices. Since every bipartite graph has no odd cycles, it follows that $G$ has no elementary figure with $n$ vertices. Again \cite[Theorem 1.3, p. 32]{CDS} implies that $\det(G)=0$.
\end{proof}
\begin{qu}
For a given positive integer $n$, find integers $m$ such that $\det(G)\not=m$ for all graphs $G$ with $n$ vertices.
\end{qu}
\begin{thm}\label{Ry}
Let $G$ be a graph with $n\geq 2$ vertices and $m$ edges. Then $|\det(G)|\leq \big(\frac{2m}{n}\big)^n\big(1-\frac{2m-n}{n(n-1)}\big)^{n-1}$.
\end{thm}
\begin{proof}
Note that the number of $1$'s in any adjacency matrix of $G$ is equal to $2m$. Now the proof  follows from Theorem 3 of \cite{R}.
\end{proof}
\begin{prop}\label{Nee}
Let $G$ be a graph with $n$ vertices and let $\{d_1,\dots,d_n\}$ be the set of vertex degrees of $G$. If $d=\gcd(d_1,\dots,d_n)$, then $d\cdot\gcd(\frac{2m}{d},d)$ divides $\det(G)$.
\end{prop}
\begin{proof}
Let $A=\begin{bmatrix}
      A_1 & \cdots & A_n
    \end{bmatrix}$
 be an adjacency matrix of $G$, where $A_1,\dots, A_n$ are its columns. Then
 $$\det(G)=\begin{vmatrix}A_1+\cdots+A_n & A_2  & \cdots & A_n                                                                                     \end{vmatrix}=\begin{vmatrix} \mathbf{d} & A_2 & \cdots & A_n \end{vmatrix},$$
 where $\mathbf{d}=\begin{bmatrix}
                     d_1 & \cdots & d_n
                   \end{bmatrix}^T$.
  Now sum up all the rows with the $n$th one. Then $\det(G)$ is equal to the determinant of a  matrix whose first column
  is $$\begin{bmatrix} d_1 & \cdots & d_{n-1} &\sum_{i=1}^nd_i\end{bmatrix}^T$$ and its $n$th row is
  $$\begin{bmatrix} \sum_{i=1}^nd_i & d_2&\cdots & d_{n}\end{bmatrix}.$$
  Since $\sum_{i=1}^n d_i=2m$, by factoring $d$ from the first column and $\gcd(\frac{2m}{d},d_2,\dots,d_n)$ from the last column, we have that
  $\det(G)=d\cdot\gcd(\frac{2m}{d},d_2,\dots,d_n)$. Since $d_1=2m-\sum_{i=2}^nd_i$, we have  $\gcd(\frac{2m}{d},d_2,\dots,d_n)=\gcd(\frac{2m}{d},d)$.
  This completes the proof.
 \end{proof}
 \begin{thm} {\rm (Hu \cite{H})}  \label{H}
 Let $G$ be a connected  graph with $n$ vertices  having a unique cycle $C$ with $k<n$ vertices. Then
$$\det(G)=\begin{cases} 1 & G \; {\rm has \; a \; perfect \; matchig},\; k\equiv 1 \mod 2, \; n\equiv 0 \mod 4\\
-1 & G \; {\rm has \; a \; perfect \; matchig},\; k\equiv 1 \mod 2, \; n\equiv 2 \mod 4\\
4 & G \; {\rm has \; a \; perfect \; matchig},\; k\equiv 2 \mod 4, \; n\equiv 0 \mod 4\\
-4 & G \; {\rm has \; a \; perfect \; matchig},\; k\equiv 2 \mod 4, \; n\equiv 2 \mod 4\\
0 & G \; {\rm has \; a \; perfect \; matchig}, \; k\equiv 0 \mod 4\\
2 & G \; {\rm has \; no \; perfect \; matchig},\; G\backslash C \; {\rm has \; a \; perfect \; matching}, n-k\equiv 0 \mod 4\\
-2 & G \; {\rm has \; no \; perfect \; matchig},\; G\backslash C \; {\rm has \; a \; perfect \; matching}, n-k\equiv 2 \mod 4\\
0 & {\rm both} \; G \; {\rm and} \; G\backslash C \; {\rm have \; no \; perfect \; matching}
\end{cases}.$$
\end{thm}
\begin{cor}
If  $\alpha_n=\det(G)$ or $\beta_n=\det(H)$  for some graphs $G$ and $H$, then neither $G$ nor $H$ have a unique cycle for all $n>5$.
\end{cor}
\begin{proof}
It is not hard to see that for $n>5$, $\alpha_n>4$ and $\beta_n<-4$. Since a graph with a unique cycle has exactly one
connected component with a unique cycle and the other components (if exist) are tree (whose determinants belong to $\{0,1,-1\}$),   Theorem \ref{H} completes the proof.
\end{proof}
\begin{lem}\label{CC}
Let $G$ be a graph containing exactly two cycles $C_1$ and $C_2$ of orders $k$ and $\ell$, respectively such that $|V(C_1) \cap V(C_2)|=1$.
If $|V(G)|=k+\ell-1$, then $$\det(G)=\begin{cases} 0 & {\rm if \; both} k,\ell \; {\rm are \; even}\\
2\cdot (-1)^{\ell+k-1}\cdot \big( (-1)^{\frac{\ell-1}{2}+1}+ (-1)^{\frac{k-1}{2}+1}\big) & {\rm if \; both}\; k,\ell \; \rm{ are \; odd}\\
 2\cdot (-1)^{\ell+k-1}\cdot \big( (-1)^{\frac{\ell}{2}+\frac{k-1}{2}} +(-1)^{\frac{k-1}{2}+1}\big) & {\rm if}\; k \; {\rm is\; odd}, \;\ell \; {\rm is \; even}\\
 2\cdot (-1)^{\ell+k-1}\cdot \big( (-1)^{\frac{k}{2}+\frac{\ell-1}{2}}+ (-1)^{\frac{\ell-1}{2}+1}\big) & {\rm if}\; \ell \; {\rm is\; odd}, \;k \; {\rm is \; even}\\
 \end{cases}$$
\end{lem}
\begin{proof}
Let $C_1:=x_1\cdots x_k x_1$ and $C_2:=y_1\cdots y_\ell y_1$. Suppose $C_1\cap C_2=\{x_k\}=\{y_\ell\}$. We want to apply \cite[Theorem 1.3, p. 32]{CDS} and so we need to find all elementary figures of $G$, that are all spaning   subgraphs of $G$ with exactly $k+\ell-1$ vertices whose connected components are either an edge or a cycle of $G$. If $k$ and $\ell$ are both even, then $G$ has no elementary figure. If $k$ and $\ell$ are both odd, then we have only two elementary figures
$$ F_1= \{C_1,y_1y_2,\dots,y_{\ell-2}y_{\ell}\} \;\;{\rm and} \;\; F_2=\{C_2,x_1x_2,\dots,x_{k-2}x_{k-1}\}.$$
If $k$ is odd and $\ell$ is even, we have exactly $3$ elementary figures as follows:
\begin{align*}\{y_1y_2,\dots,y_{\ell-1}y_{\ell},x_1x_2,\dots,x_{k-1}x_{k-2}\},\{y_1y_{\ell},y_2y_3,\dots,y_{\ell-2}y_{\ell-1},
 x_{k-1}x_{k-2},\dots,x_2x_1\},\\ \{C_2,x_1x_2,\dots,x_{k-2}x_{k-1}\}.
   \end{align*}
   Similarly, for the case $k$  even and $\ell$ odd, we have exactly $3$ elementary figures (interchange $k$ and $\ell$ in the latter elementary figures). Now we can apply   \cite[Theorem 1.3, p. 32]{CDS} and this completes the proof.
\end{proof}
\begin{lem}\label{C-C}
Let $G$ be a connected graph containing exactly two cycles $C_1$ and $C_2$ of orders $k$ and $\ell$, respectively such that $V(C_1)\cap V(C_2)=\varnothing$. If $V(G)=V(C_1)\cup V(C_2)$, then $\det(G)\in\{-8,0,3,5,16\}$. If $C_1$ and $C_2$ are connected by a path $P$ of length  at least $2$ and $V(G)=V(C_1)\cup V(C_2) \cup V(P)$, then
$$\det(G)\in \begin{cases}\{0,\pm 4\} & \text{if}\; t=3\\
                           \{-16,-8,-5,-3,0\} & {\rm if}\; t=4\\
                           \{0,\pm3,\pm4,\pm5,\pm8,\pm16\} & {\rm if} \; t>4   \end{cases}.$$
\end{lem}
\begin{proof}
Let $C_1=x_1\cdots x_k x_1$, $C_2=y_1\cdots y_\ell y_1$. By hypothesis, in any case, we have that there exists a path $P=x_k=z_1z_2\cdots z_t=y_{\ell}$  connecting $C_1$ and $C_2$ such that $V(G)=V(C_1)\cup V(C_2) \cup V(P)$. It is clear that
$t=2$ if and only if $V(G)=V(C_1) \cup V(C_2)$. In this case, by \cite[Theorem 2.12, p. 59]{CDS}, we have $\det(G)=\det(C_1)\det(C_2)-\det(P_{k-1})\det(P_{\ell-1})$. Now it follows from Lemma \ref{KPC}, $\det(G)\in\{-8,0,3,5,16\}$. \\
Now assume that $t\geq 3$, that is the length of $P$ is at least $3$,  and let $H$ be the induced subgraph of $G$ on the vertices $V=V(C_1) \cup \big(V(P)\backslash\{z_t\}\big)$.
Note that, by  \cite[Theorem 2.11, p. 59]{CDS}, we have
$$\det(H)=\begin{cases} -\det(P_{k-1}) & {\rm if} \; t=3\\
\det(C_k)\det(P_{t-2})-\det(P_{k-1})\det(P_{t-3}) & {\rm if} \; t\geq 4 \end{cases}.$$
Thus it follows from  \cite[Theorem 2.12, p. 59]{CDS} that
$$\det(G)=\det(H)\det(C_\ell)-\det(H\backslash\{z_{t-1}\})\det(P_{\ell-1}).$$
Note that $\det(H\backslash\{z_{t-1}\})$ can be similarly compute by a formulae as given for $\det(H)$. Now it is easy  to complete the proof.
\end{proof}
\begin{prop}\label{CC2}
Let $G$ be a graph with exactly two cycles. Then $\det(G)\in\{0,\pm 1,\pm 2, \pm 3, \pm 4,\pm 5,\pm 8,\pm 16\}$.
\end{prop}
\begin{proof}
   By using \cite[Theorem 2.11, p. 59]{CDS} on the (possible) vertices of $G$ with degree $1$, we find  that there exist a (possibly empty) forest $F$ and either a connected graph $H$ as  of the form in Lemmas \ref{CC} or \ref{C-C} or   two connected graphs $H_1$ and $H_2$ each of which contains exactly one cycles such that $\det(G)=\pm\det(F)\cdot \det(H)$ or $\det(G)=\pm \det(F) \cdot \det(H_1) \cdot \det(H_2)$, respectively, where $\det(F)=1$ if $F$ is empty. Now Lemmas \ref{CC} and \ref{C-C} and Theorem \ref{H} complete the proof.
\end{proof}
Let us  end the paper by the following problems and questions mainly arising from the data table given in Proposition \ref{t} and some other investigations.\\

\noindent{\bf Problems.}\\
1) \; Describe $\mathcal{DG}_n=\big\{\det(G) \;|\; G\in \mathcal{G}_n\big\}$.\\
2) \; What are the maximum $\alpha_n$ and minimum $\beta_n$? \\
3) \; For a given $n$, what integers can never belong to $\mathcal{D}_n$? \\
4) \; Is it possible to determine graphs $G,H\in \mathcal{G}_n$ such that $\det(G)=\alpha_n$ and $\det(H)=\beta_n$?
Do they have some distinguished properties from other graphs with $n$ vertices? For example,  must they always be connected?\\
5) \; Find relations between $\mathcal{DG}_n$ and $\mathcal{DG}_{n+1}$.\\
6) \; Is it true that $\alpha_n<\alpha_{n+1}$ for all $n>3$?\\
7) \; Is it true that $\beta_{n+1}<\beta_{n}$?\\
8) \; Is it true that $|\beta_n|>\alpha_n$ for all $n>7$? \\
9) \; Does $\delta_0:=\underset{n\rightarrow \infty}{\lim} \frac{|\{G\in\mathcal{G}_n \;|\; \det(G)=0\}|}{|\mathcal{G}_n|}$ exist?
If so, is it true that $\delta_0=\frac{1}{2}$?\\

\end{document}